\documentclass[12pt]{amsart}
\usepackage[backref=page]{hyperref}
\usepackage[msc-links, alphabetic]{amsrefs}

\usepackage{amssymb, hyperref,xcolor, mathrsfs, cancel, fullpage, comment, bbm, euscript, yfonts, longtable,lscape, blindtext}
\usepackage[capitalise]{cleveref}

\newtheorem{theorem}{Theorem}[section]
\newtheorem{lemma}[theorem]{Lemma}
\newtheorem{condition}[theorem]{Condition}
\newtheorem{proposition}[theorem]{Proposition}

\newtheorem{corollary}[theorem]{Corollary}

\theoremstyle{remark}

\numberwithin{equation}{section}

\newcommand{\ecc}{\ensuremath{\textswab{h}}}
\newcommand{\edge}{\ensuremath{\mathrm{edge}}}
\newcommand{\vcc}{\ensuremath{h}}
\newcommand{\bbC}{\ensuremath{\mathbb{C}}}
\newcommand{\mo}{{-1}}
\newcommand{\calA}{\ensuremath{\mathcal A}}
\newcommand{\calB}{\ensuremath{\mathcal B}}
\newcommand{\calG}{\ensuremath{\mathcal{G}}}
\newcommand{\scrO}{\ensuremath{\mathscr{O}}}
\newcommand{\scrV}{\ensuremath{\mathscr{V}}}
\newcommand{\vol}{\ensuremath{\mathrm{vol}}}
\newcommand{\supp}{\ensuremath{\mathrm{supp}}}
\newcommand{\bbR}{\ensuremath{\mathbb{R}}}
\newcommand{\str}{\ensuremath{\mathrm{str}}}
\newcommand{\sml}{\ensuremath{\mathrm{sml}}}
\newcommand{\calO}{\ensuremath{\mathcal{O}}}
\newcommand{\calI}{\ensuremath{\mathcal I}}

\title{On the spectral gap of Cayley graphs}

\author{Jyoti Prakash Saha}
\address{Department of Mathematics, Indian Institute of Science Education and Research Bhopal, Bhopal Bypass Road, Bhauri, Bhopal 462066, Madhya Pradesh, India}
\curraddr{}
\email{jpsaha@iiserb.ac.in}
\thanks{}
\date{\today}


\begin{document}

\maketitle

\begin{abstract}
Let $\Gamma$ be a Cayley graph, or a Cayley sum graph, or a twisted Cayley graph, or a twisted Cayley sum graph, or a vertex-transitive graph. Suppose  $\Gamma$ is undirected and non-bipartite. Let $\mu$ (resp. $\mu_2$) denote the smallest (resp. the second largest) eigenvalue of the normalized adjacency operator of $\Gamma$, and $d$ denote the degree of $\Gamma$. We show that $1+ \mu = \Omega((1-\mu_2)/d)$ holds. 
\end{abstract}

\section{Introduction}

The discrete Cheeger--Buser inequality, established by Dodziuk \cite{DodziukDifferenceEqnIsoperimetricIneq}, Alon and Milman \cite{AlonMilmanIsoperiIneqSupConcen}, and Alon \cite{AlonEigenvalueExpand}, states that the gap between the largest and the second largest eigenvalue $\mu_2$ of the normalized adjacency operator of a $d$-regular graph $\Gamma$, and its edge Cheeger constant $\ecc_\Gamma$ satisfy 
$$
2\ecc_\Gamma 
\geq 1 - \mu_2 
\geq \frac {\ecc_\Gamma^2}2.$$
The dual Cheeger inequality, established by Trevisan \cite[Equation (9)]{TrevisanMaxCutTheSmallestEigenVal} and Bauer and Jost \cite[p. 804--805]{BauerJostBipartiteNbdGraphsSpecLaplaceOp} using a technique due to Desai and Rao \cite{DesaiRaoCharacSmallestEigenvalue}, provides a similar bound relating the gap between $-1$ and the smallest eigenvalue $\mu_n$ of the normalized adjacency operator of $\Gamma$ and the edge bipartiteness constant $\beta_{\edge, \Gamma}$ of $\Gamma$ through the inequality 
$$
2\beta_{\edge, \Gamma} \geq 1 + \mu_n
\geq 
\frac{\beta_{\edge, \Gamma}^2}2.
$$
Breuillard, Green, Guralnick and Tao showed qualitatively that for Cayley graphs, combinatorial expansion implies two-sided spectral expansion \cite[Appendix E]{BGGTExpansionSimpleLie}. Biswas obtained a quantitative bound on the lower spectral gap of non-bipartite, connected Cayley graphs \cite{BiswasCheegerCayley}. Moorman, Ralli and Tetali improved the bound due to Biswas. They proved that the smallest non-trivial eigenvalue of the normalized adjacency operator of a Cayley graph $\Gamma$ of degree $d$ is $-1 + \Omega(\frac{h_\Gamma^2}{d^2})$ where $\vcc_\Gamma$ denotes the vertex Cheeger constant of $\Gamma$ \cite[Theorem 2.6]{CayleyBottomBipartite}. Recently, Hu and Liu \cite[Theorem 1.3]{HuLiuVertexIsoperimetrySignedGraph} proved a bound of the form $\Omega (\frac{\vcc_\Gamma^2}d)$ on the lower spectral gap of non-bipartite Cayley sum graphs and non-bipartite vertex-transitive graphs of degree $d$. Further, Li and Liu \cite{LiLiuNonTrivExtremalEigenvalues} showed that the lower spectral gap of a finite, non-bipartite, vertex-transitive graph $\Gamma$ of degree $d$ is at least $\min \{ \frac 2d, \frac {(\sqrt 3-1)^2} 8 \ecc_\Gamma^2\}$. The lower spectral gap for Cayley graphs, twisted Cayley graphs, twisted Cayley sum graphs, vertex-transitive graphs were also studied by Biswas and the author \cite{CheegerCayleySum,CheegerTwisted,VertexTra}, Hu and Liu \cite{HuLiuVertexIsoperimetrySignedGraph}, Li and Liu \cite{LiLiuNonTrivExtremalEigenvalues}, and the author \cite{LowerSpecGap}. 

In \cite[Appendix E]{BGGTExpansionSimpleLie}, Breuillard, Green, Guralnick and Tao used a ``pivoting technique'' due to Fre\u{\i}man \cite{FreimanGroupsInverseProb} to prove that the lower spectral gap of any undirected, non-bipartite Cayley graph is bounded from the below in terms of the Cheeger constant and the degree. This technique has been used crucially in all the above-mentioned works on the lower spectral gap that use ``combinatorial'' arguments, except the recent work of Li and Liu \cite{LiLiuNonTrivExtremalEigenvalues}, which uses ``spectral'' arguments. An important step in the combinatorial proofs is to consider a suitable dichotomy. When the dichotomy holds, one applies Fre\u{\i}man's technique and produces an index two subgroup of the underlying group, whose cosets would form a bipartition of the graph if the lower spectral gap were ``small'', and consequently, the lower spectral gap has to be ``large''. Depending on the context, one proves that a suitable dichotomy has to hold, or that the lower spectral gap is indeed ``large'' in the absence of a dichotomy. 

Moorman, Ralli and Tetali introduced a new combinatorial approach to obtain lower bounds on the lower spectral gap by controlling the bipartiteness constant. First, using Fre\u{\i}man's technique, they proved that the bipartiteness constant of a non-bipartite, connected Cayley graph is larger than a quantity depending on the Cheeger constant and the degree of the graph. Next, they used the dual Cheeger inequality to show that the lower spectral gap of such a graph is larger than a quantity depending on the Cheeger constant and the degree of the graph. 

The aim of this article is to show that for a connected, non-bipartite Cayley graph $\Gamma$, its \emph{lower spectral gap}, i.e., the gap between between $-1$ and the smallest eigenvalue of the normalized adjacency operator $A_\Gamma$ of $\Gamma$ is bounded from the below in terms of the \emph{spectral gap}, i.e., the gap between $1$ and the second smallest eigenvalue of $A_\Gamma$, and the degree of $\Gamma$. We use the strategy of Fre\u{\i}man crucially, together with suitable refinements of some of the techniques and counting arguments from \cite{BGGTExpansionSimpleLie, CayleyBottomBipartite}, and from the subsequent work \cite{VertexTra} to prove the following result. 

\begin{theorem}
\label{Thm:MainIntro}
Let $\Gamma$ be a Cayley graph, or a Cayley sum graph, or a twisted Cayley graph, or a twisted Cayley sum graph, or a vertex-transitive graph. Suppose  $\Gamma$ is undirected, connected and non-bipartite. Let $\mu$ (resp. $\mu_2$) denote the smallest (resp. the second largest) eigenvalue of the normalized adjacency  operator of $\Gamma$, and $d$ denote the degree of $\Gamma$. 
Then 
$$1 + \mu\geq \frac {1- \mu_2} {50000d}
$$
holds. Moreover, if the underlying group of $\Gamma$ is a finite simple group or it is of odd order, then 
$$1 + \mu\geq \frac {1- \mu_2} {2525}$$
holds. 
\end{theorem}
Using a result of Alon \cite{AlonEigenvalueExpand}, we obtain the corollary below. For Cayley graphs, this provides the same bound obtained in \cite{CayleyBottomBipartite}, although with weaker absolute constants. The bound provided by \cite[Corollary 1]{LiLiuNonTrivExtremalEigenvalues} for vertex-transitive graphs is stronger than the following bound.

\begin{corollary}
Let $\Gamma, d, \mu$ be as in \cref{Thm:MainIntro}. Let $\vcc_\Gamma$ denote the vertex Cheeger constant of $\Gamma$. Then the bound 
$$1 + \mu\geq 
\frac {\vcc_\Gamma^2}{350000 d^2}
$$
holds. Moreover, if the underlying group of $\Gamma$ is a finite simple group or it is of odd order, then 
$$1 + \mu\geq \frac {\vcc_\Gamma^2} {17675}$$
holds. 
\end{corollary}

The bounds provided by \cref{Thm:MainIntro} are stronger than those obtained in \cite{BiswasCheegerCayley, CheegerCayleySum,CheegerTwisted,VertexTra}. Moreover, the bounds stated in \cref{Thm:MainIntro} cannot be deduced from the bounds obtained in \cite{CayleyBottomBipartite, LowerSpecGap, HuLiuVertexIsoperimetrySignedGraph}, and vice versa. 

\section{Notations}

Let $V$ be a finite set with $|V|\geq 2$. 
Let $\ell^2(V)$ denote the Hilbert space of functions $f: V \to \bbC$ equipped with the inner product 
$$\langle f, g\rangle : = \sum_{v\in V} f(v) \overline {g(v)}$$
and the norm 
$$||f||_{\ell^2(V)} 
: = 
\sqrt{\langle f, f\rangle}.$$
A permutation $\rho: V \to V$ induces an operator $P_\rho:\ell^2(V) \to \ell^2(V)$ defined by 
$$(\rho\cdot f)(v) : = f(\rho^\mo v)
.$$
If $V$ carries a left action of a group $\calG$, then the induced action of $\calG$ on $\ell^2(V)$ is defined by 
$$(\tau\cdot f)(v) : = f(\tau^\mo v).$$

\begin{condition}
\label{SelfAdjConnNonBipartiteRegular}
\quad 
\begin{enumerate}
 \item 
 Let $V$ be a finite set with $|V|\geq 2$.
 \item 
 Let $T: \ell^2(V) \to \ell^2(V)$ be a self-adjoint operator with  $d\mu_2$ denoting its second largest eigenvalue.  
 \item Assume that $\mu_2 \neq 1$. 
\item Suppose 
$a_{uv}:=
\langle 
T1_v, 1_u 
\rangle$ 
is a non-negative integer for any $u, v\in V$, and $T$ has degree $d >0$, i.e., $\langle 
T1_V, 1_u 
\rangle = d$ for any $u\in V$. 
\end{enumerate}
\end{condition}

In the following, $1_V$ denotes the constant function on $V$ taking the value $1$ everywhere. Denote the function $\frac 1{\sqrt n} 1_V $ by $u$ where $n = |V|$. 
For a subset $F$ of $V$, define 
$$
\vol_T (F) 
:= 
\langle T 1_V, 1_{F} \rangle
.$$
The edge Cheeger  constant of $T$, denoted by $\ecc_T$,  is defined by  
$$
\ecc_T 
:=
\min _{F \subseteq V, F \neq \emptyset, V }
\frac
{
\langle T 1_F, 1_{V\setminus F} \rangle
}
{
\min 
\{
\vol_T(F), \vol_T(V\setminus F)
\}
}.
$$
The edge bipartiteness constant of $T$, denoted by $\beta_{\edge, T}$, is defined by 
$$
\beta_{\edge, T} 
:=  \min_{ \psi : V \to \{-1, 0, 1\}, \psi \neq 0} \frac { \langle (dI + T) \psi, \psi \rangle }{2d\| \psi \|_2^2} .$$

\begin{condition}
\label{EigenFunction}
Let $f$ be a real-valued eigenfunction of $T$ with eigenvalue $d\mu$. 
\begin{enumerate}
\item 
The eigenfunction $f$ is orthogonal to $1_V$. 

\item The eigenfunction $f$ is nonzero everywhere on $V$. 
\item The inequality $|\supp(f_+)| \geq |\supp (f_-)|$ holds. 
\item 
Assume that $\mu \neq -1$ and $1 + \mu < 1 - \mu_2$, i.e., $0 <\varkappa < 1$ where $\varkappa:= \frac {1+ \mu}{1- \mu_2}$. 
\end{enumerate}
\end{condition}
Henceforth, we assume that \cref{SelfAdjConnNonBipartiteRegular} and \cref{EigenFunction} hold. 
Note that $\mu\neq 1$, otherwise, $1 + \mu = 2 \geq 1 - \mu_2$, which contradicts $1+ \mu < 1- \mu_2$. Since $d\mu_2$ is the second largest eigenvalue of $T$, it follows that $\mu\leq \mu_2$. 
Note that if $\mu_2\leq 0$, then $\mu\leq \mu_2$ implies that $\mu \leq 0$. If $\mu_2 >0$, then $1+ \mu < 1- \mu_2$ implies that $\mu<0$. 

\begin{theorem}
\label{Thm:DualCheegerNewUpp}
The inequalities
$$
\beta_{\edge, T}
\leq \frac { 1+ \mu} 2 \frac 1{1 - \varkappa}, 
\quad
\varkappa 
\geq \frac{\beta_{\edge, T}}2
$$
hold. 
\end{theorem}

By the dual Cheeger inequality, it follows that $\beta_{\edge, T} \geq \frac{1 + \mu}2$. The above result provides an upper bound on $\beta_{\edge, T}$ in terms of $\varkappa$ and the gap $1+ \mu$. It follows from \cref{Lemma:LowerBddForKappaGeneral}.

\begin{theorem}
\label{Thm:GenFirst}
Suppose \cref{SelfAdjConnNonBipartiteRegular},  \cref{EigenFunction}, \cref{Condition22}, and \cref{Left2Right} hold. Also assume that \cref{ConditionConn} holds with $\nu \geq \frac 12$. Then the inequality 
$$
1+ \mu 
\geq 
\frac {1- \mu_2}{50000 d}
$$
holds. 
\end{theorem}

Define 
$$
\varsigma 
= \frac{\sqrt{\varkappa}}{\sqrt{1- \varkappa}}, \quad 
\vartheta
= 
\frac 1{\sqrt 2}\sqrt{ 1 - \varkappa}
- 
\sqrt{\varkappa}, \quad 
\delta = 
\frac \xi2 
\left(
\frac {3\left(1 - \varkappa\right)} {1  + \varsigma + \varkappa + 2 \sqrt 2 \sqrt{\varkappa}}
\left(
\frac 1{\sqrt 2}
-
\varsigma 
\right) ^2 
-1
\right)
$$
where $\xi$ is an element of $[0, 1]$ to be specified later. 

\begin{proof}
[Proof of \cref{Thm:MainIntro}]

Note that 
\cref{Lemma:NowhereVanishingEigenfunctions} shows that \cref{EigenFunction} holds for Cayley graphs, Cayley graphs, twisted Cayley graphs, twisted Cayley sum graphs and vertex-transitive graphs. 
These graphs also satisfy \cref{SelfAdjConnNonBipartiteRegular} whenever they are undirected, connected and non-bipartite. 
\cref{Condition22}, and \cref{Left2Right} hold for these graphs. Moreover, \cref{ConditionConn} holds for them with $\nu \geq \frac 12$. This proves the first bound. The second bound follows from 
\cref{Lemma:SubgrpOfIndex2} and 
\cref{Lemma:LackOfDichotomyGivesLowerSpecGap}. 
\end{proof}

We provide an outline of the proof of \cref{Thm:GenFirst}. In \cref{Lemma:DecompositionStrSml}, we show that an eigenfunction of $T$, with eigenvalue ``close to $-1$'', admits an orthogonal decomposition into a ``structured'' component with ``large'' norm and a ``pseudorandom'' component with ``small'' norm. Next, in \cref{Lemma:fsmlL1bdd}, we obtain bounds on the $\ell^1$-norm of $f$ restricted to certain subsets. 
In  \cref{Lemma:DiffOfL2NormOFf+f-}, \cref{Lemma:ControllingL2NormFromBelow}, we bound the $\ell^2$-norm of $f_+$ in terms of the $\ell^2$-norm of $f$ and $\varkappa$. Assuming the dichotomy that for any $\tau \in \calG$, either $f_+, (\tau \cdot f)_+$  correlate substantially, or their correlation is negligible, we prove the existence of an index two subgroup of $\calG$ in \cref{Lemma:SubgrpOfIndex2} using an argument due to Fre\u{\i}man \cite{FreimanGroupsInverseProb}. 
Under this dichotomy, we show in \cref{Lemma:DichotomyYieldsLowerBdd} that 
$$
1+ \mu 
\geq \frac {1- \mu_2} {50000 d}
$$
holds. 
If such a dichotomy does not hold, then we show in 
\cref{Lemma:LackOfDichotomyGivesLowerSpecGap}
that 
$$1+ \mu 
\geq \frac {1-\mu_2}{55000} .$$
Combining these bounds, \cref{Thm:GenFirst} follows.

\section{A lower bound on $\varkappa$}

\begin{lemma}
[Decomposition]
\label{Lemma:DecompositionStrSml}
Let $f_\str, f_\sml$ be the functions on $V$ defined by 
\begin{align*}
f_\str & : =  \langle f , \frac 1{\sqrt n} (1_{f> 0} - 1_{f< 0} ) \rangle \frac 1{\sqrt n} ( 1_{f> 0} - 1_{f< 0} ), \\
f_\sml &:= f - f_\str.
\end{align*}
The inequalities 
\begin{align*}
\| f_\sml \|_2 & \leq \sqrt\varkappa \|f\|_2, \\
\| f_\str \|_2 & \geq \sqrt{1 - \varkappa }\|f\|_2, \\
\| f_\sml \|_2 & \leq \varsigma \| f_\str \|_2
\end{align*}
hold. 
\end{lemma}

\begin{proof}
Define 
\begin{align*}
|f|_\str & := \langle |f| , u\rangle u, \\
|f|_\sml & : = |f| - |f|_\str. 
\end{align*}
We obtain bounds for the $\ell^2$-norms of $|f|_\str, |f|_\sml$ following \cite[Lemma 3]{LiLiuNonTrivExtremalEigenvalues}. 
Since $|f|_\sml$ is orthogonal to $u$, 
it follows that 
$$
\langle T |f|_\sml , |f|_\sml \rangle 
\leq d\mu_2 \||f|_\sml\|_2^2,$$
which yields 
\begin{align*}
d (1 - \mu_2)  \||f|_\sml\|_2^2
& \leq d\||f|_\sml\|_2^2 - \langle T |f|_\sml , |f|_\sml \rangle \\
& =  \langle (dI -T) |f|_\sml , |f|_\sml \rangle \\
& =  \langle (dI -T) (|f| -  |f|_\str) , |f|_\sml \rangle \\
& =  \langle (dI -T) |f| , |f|_\sml \rangle \\
& =  \langle (dI -T) |f| , |f| - |f|_\str \rangle \\
& =  \langle (dI -T) |f| , |f| \rangle  - \langle (dI -T) |f| , |f|_\str \rangle\\
& =  \langle (dI -T) |f| , |f| \rangle  - \langle  |f| , (dI -T)|f|_\str \rangle 
\\
& =  \langle (dI -T) |f| , |f| \rangle   \\
& = d\sum_u f(u)^2 - \sum_{u, v\in V \times V} a_{uv} |f(u) || f(v) |\\
& \leq d\sum_u f(u)^2 + \sum_{u, v\in V \times V} a_{uv} f(u) f(v) \\
& =  \langle (dI +T)f, f\rangle   \\
& =  \langle (d+d\mu) f , f\rangle   \\
& = d(1+\mu) \|f\|_2^2 .
\end{align*}
This shows that 
\begin{align*}
\| |f|_\sml \|_2^2 & \leq \varkappa \|f\|_2^2, \\
\| |f|_\str \|_2^2 & \geq \left( 1 - \varkappa\right)\|f\|_2^2.
\end{align*}
Note that 
\begin{align*}
f_\str &  =  \langle f , \frac 1{\sqrt n} (1_{f> 0} - 1_{f< 0} ) \rangle \frac 1{\sqrt n} ( 1_{f> 0} - 1_{f< 0} ) \\
& =  \frac 1 { n}  \langle f, 1_{f> 0 } - 1_{f< 0}\rangle (1_{f> 0 } - 1_{f< 0}) \\
& =  \frac 1 { n}  \langle |f|,  1_V \rangle (1_{f> 0 } - 1_{f< 0}), \\
|f|_\str 
& =  \frac 1 { n}  \langle |f|,  1_V \rangle 1_V \\
& = \frac 1 { n}  \langle |f|,  1_V \rangle (1_{f> 0 } + 1_{f< 0}) .
\end{align*}
It follows that 
$$f_\str = \frac f{|f|} |f|_\str.$$
We also have 
$$
f_\sml = f - f_\str = f - \frac f{|f|} |f|_\str = \frac f{|f|} |f| - \frac f{|f|} |f|_\str = \frac f{|f|} |f|_\sml.
$$
The bounds follow. 
\end{proof}

The \emph{bipartiteness constant} of $f$ is defined by 
$$
\beta: = \beta_f 
:= 
\frac {\langle T1_{f>0} , 1_{f>0} \rangle + \langle T1_{f<0} , 1_{f<0} \rangle }{dn}
.
$$
The following provides a lower bound on $\varkappa = \frac{1+ \mu}{1-\mu_2}$ in terms of the bipartiteness constant $\beta$ of $f$. 
By the dual Cheeger inequality, it follows that $2\beta \geq 1 + \mu$. We prove that $2\beta \leq \frac{1 + \mu} {1 - \varkappa}$ if $\varkappa < 1$. Noting that $\beta \geq \beta_{\edge, T}$, \cref{Thm:DualCheegerNewUpp} follows from \cref{Lemma:LowerBddForKappaGeneral} below.

\begin{theorem}
\label{Lemma:LowerBddForKappaGeneral}
The inequality 
$$
\beta
\leq 
\frac{1+ \mu  + (1- \mu) \varsigma^2 }2
=\frac{1-\mu_2} 2 \frac {\varkappa}{1 - \varkappa}
= \frac { 1+ \mu} 2 \frac 1{1 - \varkappa}
\leq \frac { \varkappa}{1 - \varkappa}
$$
holds. Consequently, 
$$\varkappa
\geq \frac{2\beta}{1-\mu_2+2\beta }
\geq \frac{\beta}{1+ \beta }
\geq \frac \beta 2
$$
holds. In particular, if $2d\beta \geq 1$, then 
$$\varkappa\geq \frac 1{d(1 - \mu_2) + 1} > \frac 1{2d+1}.$$
\end{theorem}

\begin{proof}
Since $f$ is an eigenvector of $T$ with eigenvalue $d\mu$, it follows that 
\begin{align*}
& 2d(1 + \mu) \|f\|_2^2\\
& = \sum_{u, v} a_{uv}(f(u) + f(v))^2 \\
& = \sum_{u, v} a_{uv}(
(f_\str(u) + f_\str(v))
+ 
(f_\sml(u) + f_\sml(v))
)^2 \\
& = 
\sum_{u, v} a_{uv} (f_\str(u) + f_\str(v))^2
+ 
\sum_{u, v} a_{uv}(f_\sml(u) + f_\sml(v))^2  \\
& + 
2\sum_{u, v} a_{uv} (f_\str(u) + f_\str(v)) (f_\sml(u) + f_\sml(v)) \\
& = 4d\beta \|f_\str\|_2^2
+ 
\sum_{u, v} a_{uv}(f_\sml(u) + f_\sml(v))^2  
 + 
2\sum_{u, v} a_{uv} (f_\str(u) + f_\str(v)) (f_\sml(u) + f_\sml(v)) \\
& = 4d\beta \|f_\str\|_2^2 + 2d\|f_\sml\|_2^2 -2 \langle A f_\sml, f_\sml\rangle 
+ 
2d\langle f_\str, f_\sml \rangle + 4\langle Af, f_\sml \rangle \\
& = 4d\beta \|f_\str\|_2^2 + 2d\|f_\sml\|_2^2 -2 \langle A f_\sml, f_\sml\rangle 
+ 4\langle d\mu f, f_\sml \rangle \\
& = 4d\beta \|f_\str\|_2^2 + 2d\|f_\sml\|_2^2 -2 \langle A f_\sml, f_\sml\rangle 
+ 4 d\mu \| f_\sml \|_2^2 \\
& = 4d\beta \|f_\str\|_2^2 + 2d(1+ \mu) \|f_\sml\|_2^2 
+ 2 d\mu \| f_\sml \|_2^2 -2 \langle A f_\sml, f_\sml\rangle \\
& \geq 4d\beta \|f_\str\|_2^2 + 2d(1+ \mu) \|f_\sml\|_2^2 
+ 2 d\mu \| f_\sml \|_2^2  - 2d \|f_\sml\|_2^2.
\end{align*}
This shows that 
\begin{align*}
2\beta
& \leq  (1+ \mu)  + (1- \mu) \frac {\| f_\sml \|_2^2 }{\|f_\str\|_2^2 } \\
& \leq  1+ \mu  + (1- \mu) \varsigma^2 \quad 
\text{(using \cref{Lemma:DecompositionStrSml})}\\
& = (1- \mu) \frac{1+ \mu} {1- \mu_2 - (1+\mu)} + 1 + \mu\\
& = (1+ \mu) \left(\frac{1- \mu} {1- \mu_2 - (1+\mu)} +1\right)\\
& = (1- \mu_2) \frac{1+\mu} {1- \mu_2 - (1+\mu)} \\
& = (1- \mu_2) \frac \varkappa {1- \varkappa}.
\end{align*}
\end{proof}

\section{Preparatory lemmas}

\subsection{Controlling the norms}

\begin{lemma}
[Bounding the $\ell^1$-norm]
\label{Lemma:fsmlL1bdd}
For any subset $X$ of $V$, the inequalities
\begin{align*}
\frac 1{\sqrt n}
\left|
\sum_{x \in \supp(f_+) \cap X}   |f_\sml(x)  |
\right|
& \leq 
\varsigma
\frac{\sqrt{|\supp(f_+)\cap X|}}{\sqrt n}
  \|f_\str\|_2 \\
\frac 1{\sqrt n} \left|\sum_{x \in \supp(f_+) \cap X } f(x) \right| 
& \geq \|f_\str\|_2 \frac{|\supp(f_+) \cap X|}{ n} 
-
\varsigma \|f_\str\|_2 
\frac{\sqrt{|\supp(f_+) \cap X|} }{\sqrt n}
\end{align*}
hold. 
\end{lemma}

\begin{proof}
Note that 
\begin{align*}
\left(
\sum_{x \in \supp(f_+) \cap X}   |f_\sml(x)  |
\right)^2 
& \leq  |\supp(f_+)\cap X|\|f_\sml 1_{\supp(f_+)\cap X}\|_2^2\\
& \leq |\supp(f_+)\cap X|\|f_\sml 1_{f>0}\|_2^2\\
& \leq |\supp(f_+)\cap X|\|f_\sml\|_2^2\\
& \leq 
\varsigma^2
|\supp(f_+)\cap X|
  \|f_\str\|_2^2
  \quad 
  \text{(using \cref{Lemma:DecompositionStrSml})}.
\end{align*}
Moreover, we have 
\begin{align*}
 \left|\sum_{x \in \supp(f_+) \cap X } f(x) \right| 
& = \left|\sum_{x \in \supp(f_+) \cap X } f_\str(x) + f_\sml(x)  \right| \\
& \geq \left|\sum_{x \in \supp(f_+) \cap X } f_\str(x) \right| - \left|\sum_{x \in \supp(f_+) \cap X }  f_\sml(x)  \right| \\
& = \langle |f|, u\rangle \frac{|\supp(f_+) \cap X|}{\sqrt n} 
- \left|\sum_{x \in \supp(f_+) \cap X }  f_\sml(x)  \right| \\
& \geq \|f_\str\|_2 \frac{|\supp(f_+) \cap X|}{\sqrt n} 
-
\varsigma \|f_\str\|_2 \sqrt{|\supp(f_+) \cap X|} .
\end{align*}
\end{proof}

\begin{lemma}
[Upper bound on the $\ell^2$-norm]
\label{Lemma:DiffOfL2NormOFf+f-}
Suppose $\varkappa$ lies in $[0,1/5]$. 
The inequalities 
$$
|\supp(f_+)|\leq \frac{1+\varsigma}2 n,
\quad 
|\supp(f_-)|\geq \frac{1-\varsigma}2 n,
\quad 
|\supp(f_-)|\geq \frac{1-\varsigma}{1+ \varsigma} |\supp(f_+)|,$$
$$
\|f_+\|_2^2 
\leq 
\frac {1  + \varsigma + \varkappa + 2 \sqrt 2 \sqrt{\varkappa}} 2 \|f\|_2^2$$
hold. 
\end{lemma}

\begin{proof}

Since $\langle f, u\rangle =0$, it follows that 
$\langle f_\str, u\rangle 
= - \langle f_\sml, u\rangle $. 
Note that 
\begin{align*}
\langle f_\str, u\rangle 
& =  
\langle \frac 1 { n}  \langle |f|,  1_V \rangle (1_{f> 0 } - 1_{f< 0}), u\rangle \\
& =  \frac 1 { n}  \langle |f|,  u \rangle 
\langle (1_{f> 0 } - 1_{f< 0}), 1_V \rangle \\
& =  \frac 1 { n}  \|f_\str\|_2 (|\supp(f_+)| - |\supp (f_-)|) \\
& \geq \frac{|\supp(f_+)| - |\supp (f_-)|}{n} \sqrt{1-\varkappa }\|f\|_2
\quad 
\text{(using \cref{EigenFunction}(3))} ,
\end{align*}
and 
\begin{align*}
\frac{|\supp(f_+)| - |\supp (f_-)|}{n}\sqrt{1 - \varkappa}\|f\|_2 
& \leq \langle f_\str , u\rangle\\
& = |\langle f_\sml , u\rangle |\\
& \leq \langle |f_\sml| , u\rangle  \\
& \leq \|f_\sml\|_2 \\
& \leq \sqrt{\varkappa }\|f\|_2
\end{align*}
hold. 
This yields 
$$\frac{|\supp(f_+)| - |\supp (f_-)|}{n}\sqrt{1 - \varkappa}
 \leq \sqrt{\varkappa }.$$
Combining the bound 
$$
|\supp(f_+)| - |\supp(f_-)| \leq \varsigma n$$
with 
$$|\supp(f_+)| + |\supp(f_-)| = n,$$
we obtain the stated bounds on the sizes of the supports of $f_+, f_-$. 
Also note that 
\begin{align*}
\|f_\sml 1_{f<0} \|_2^2 
& \leq \|f_\sml\|_2^2\\
& \leq \varsigma^2  \|f_\str\|_2^2 
\quad 
\text{(using \cref{Lemma:DecompositionStrSml})},\\
\|f_\str 1_{f<0}\|_2^2 
& = \frac{\langle |f|, u\rangle^2 }{ n}  |\supp(f_-)|\\
& = \|f_\str\|_2^2 \frac {|\supp(f_-)|} n \\
& \geq  \|f_\str\|_2^2 \frac {1 - \varsigma} 2
\end{align*}
hold. 
Since $\varkappa$ lies in $[0, 1/5]$, we have $\varsigma \in [0, 1/2]$, which implies that 
$\frac {1-\varsigma}2 \geq \varsigma^2$, 
and hence 
\begin{equation}
\label{Eqn:StrSmlBdd1}
\|f_\str 1_{f<0}\|_2 
\geq 
\|f_\sml 1_{f<0} \|_2 
.
\end{equation}
Also note that 
\begin{align*}
\|f_\sml 1_{f>0} \|_2^2 
& \leq \|f_\sml\|_2^2\\
& \leq \varkappa  \|f\|_2^2 \\
& \leq  \frac {\varkappa }{1 - \varkappa }  \|f_\str\|_2^2 
\quad 
\text{(using \cref{Lemma:DecompositionStrSml})},\\
\|f_\str 1_{f>0}\|_2^2 
& = \frac{\langle |f|, u\rangle^2 }{ n}  |\supp(f_+)|\\
& = \|f_\str\|_2^2 \frac {|\supp(f_+)|} n \\
& \geq  \frac 12 \|f_\str\|_2^2 
\end{align*}
hold. 
Since $\varkappa$ lies in $[0, 1/3]$, we obtain 
$$
\|f_\str 1_{f>0}\|_2 
\geq 
\|f_\sml 1_{f>0} \|_2 
.$$
Moreover, the following
\begin{align*}
\|f_+\|_2^2
& = \|f 1_{f>0}\|_2^2 \\
& = \|f_\str 1_{f>0} + f_\sml 1_{f>0} \|_2^2 \\
& \leq (\|f_\str 1_{f>0}\|_2 + \|f_\sml 1_{f>0} \|_2)^2 ,\\
\|f_-\|_2^2
& = \|f 1_{f<0}\|_2^2 \\
& = \|f_\str 1_{f<0} + f_\sml 1_{f<0} \|_2^2 \\
& \geq (\|f_\str 1_{f<0}\|_2 - \|f_\sml 1_{f<0} \|_2)^2
\quad 
\text{(using \cref{Eqn:StrSmlBdd1})}  
\end{align*}
hold. This yields 
\begin{align*}
& \|f_+\|_2^2 - \|f_-\|_2^2\\
& \leq 
(\|f_\str 1_{f>0}\|_2 + \|f_\sml 1_{f>0} \|_2)^2 - (\|f_\str 1_{f<0}\|_2 - \|f_\sml 1_{f<0} \|_2)^2 \\
& \leq 
\|f_\str 1_{f>0}\|_2^2 + \|f_\sml 1_{f>0} \|_2^2 
+ 2\|f_\str 1_{f>0}\|_2 \|f_\sml 1_{f>0} \|_2 \\
& - \|f_\str 1_{f<0}\|_2^2 -   \|f_\sml 1_{f<0} \|_2^2 + 2 \|f_\str 1_{f<0}\|_2\|f_\sml 1_{f<0} \|_2 \\
& = 
\|f_\str \|_2^2 \frac{|\supp(f_+)|} n - \|f_\str \|_2^2 \frac{|\supp(f_-)|} n + \|f_\sml 1_{f>0} \|_2^2  -   \|f_\sml 1_{f<0} \|_2^2 \\
& + 2\|f_\str 1_{f>0}\|_2 \|f_\sml 1_{f>0} \|_2 
 + 2 \|f_\str 1_{f<0}\|_2\|f_\sml 1_{f<0} \|_2 \\
& \leq 
\|f_\str \|_2^2 \frac{|\supp(f_+)| - |\supp(f_-)|} n + \|f_\sml 1_{f>0} \|_2^2  -   \|f_\sml 1_{f<0} \|_2^2 \\
& + 2 (\|f_\str 1_{f>0}\|_2 
 +  \|f_\str 1_{f<0}\|_2) \|f_\sml \|_2 \\
& = 
\|f_\str \|_2^2 \frac{|\supp(f_+)| - |\supp(f_-)|} n + \|f_\sml 1_{f>0} \|_2^2  -   \|f_\sml 1_{f<0} \|_2^2 \\
& + 2 \frac{\sqrt{|\supp(f_+)|} + \sqrt{|\supp(f_-)|}}{\sqrt n} \|f_\str \|_2 \|f_\sml \|_2 \\
& \leq 
\left(
 \frac{|\supp(f_+)| - |\supp(f_-)|} n + \varkappa 
 + 2 \frac{\sqrt{|\supp(f_+)|}+\sqrt{|\supp(f_-)|} }{\sqrt n} \sqrt{\varkappa  }
 \right)
 \|f \|_2^2 \\
& \leq 
\left(
\varsigma + \varkappa 
 + 2 \sqrt 2\sqrt{\varkappa  }
 \right)
 \|f \|_2^2 .
\end{align*}
Using $ \|f_+\|_2^2 + \|f_-\|_2^2 = \|f\|_2^2$, we obtain  the bound on $\|f_+\|_2^2$. 
\end{proof}

\begin{lemma}
[Lower bound on the $\ell^2$-norm]
\label{Lemma:ControllingL2NormFromBelow}
Suppose $\varkappa$ lies in $[0,1/3]$. 
The bounds 
\begin{align*}
\|f_+\|_2 & \geq \vartheta \|f\|_2, \\
\|f_-\|_2 
& \leq \sqrt {1 - \vartheta^2} \|f\|_2  
\leq \sqrt{\vartheta^{-2} -1} \|f_+\|_2 
\end{align*}
hold. 
\end{lemma}

\begin{proof}
Note that 
\begin{align*}
\left( 1 - \varkappa\right)\|f\|_2^2
& \leq \| f_\str \|_2^2
\quad 
\text{(using \cref{Lemma:DecompositionStrSml})}\\
& = \frac {\langle |f|, u\rangle^2}n |V| \\
& \leq 2 \frac {\langle |f|, u\rangle^2}n |\supp(f_+)| \\
& = 2 \|f_\str 1_{f>0} \|_2^2 \\
& = 2 \|f_+ - f_\sml 1_{f>0} \|_2^2 \\
& \leq 2 (\|f_+\|_2 + \|f_\sml 1_{f>0} \|_2)^2 \\
& \leq 2 (\|f_+\|_2 + \|f_\sml \|_2)^2 
\end{align*}
hold, which implies that 
\begin{align*}
\frac 1{\sqrt 2}\sqrt{ 1 - \varkappa}\|f\|_2
& \leq  \|f_+\|_2 + \|f_\sml \|_2\\
& \leq  \|f_+\|_2 + \sqrt{\varkappa}\|f \|_2
\quad 
\text{(using \cref{Lemma:DecompositionStrSml})}.
\end{align*}
This yields the bound 
$
\vartheta 
\|f\|_2
 \leq  \|f_+\|_2$. 
Using $\|f\|_2^2 = \|f_+\|_2^2 + \|f_-\|_2^2$ and $0 \leq \vartheta \leq 1$ for $\varkappa \in [0, \frac 13]$, the remaining bounds follow. 
\end{proof}

\subsection{Comparison along translates and two partitions}

For a permutation $\pi: V \to V$ and a subset $\scrV$ of $V$, define 
$$\calI_{\pi, \scrV }  :=\supp(1_\scrV  1_{\pi\cdot 1_\scrV }) =  \scrV  \cap \pi (\scrV ).$$

\begin{lemma}
\label{Lemma:InnerProdOfPlusPartWithPlusPartOfATranslate}
Let $\pi : V \to V$ be a permutation. 
The inequalities 
\begin{align*}
\left|
\langle f_+, (\pi \cdot f)_+\rangle 
- 
\left\langle 
\frac{\|f_\str\|_2}{\sqrt n} 1_{f>0} , \frac{\|f_\str\|_2}{\sqrt n} 1_{\pi \cdot f>0}
\right\rangle 
\right| 
& = 
\left|
\langle f_+, (\pi \cdot f)_+\rangle - \frac{|\calI_{\pi, f>0} |}n \|f_\str\|_2^2 
\right|\\
& \leq 
\left(
2\frac{\sqrt{|\calI_{\pi, f>0} |}}{\sqrt n}
\sqrt \varkappa 
+ 
\varkappa 
\right) 
\|f\|_2^2 , \\
\left|
\langle f_+, (\pi \cdot f)_+\rangle - \frac{|\calI_{\pi, f>0} |}n \|f\|_2^2 
\right|
& \leq 
\left(  \frac{|\calI_{\pi, f>0} |}n \varkappa+ 
2\frac{\sqrt{|\calI_{\pi, f>0} |}}{\sqrt n}
\sqrt \varkappa 
+ 
\varkappa 
\right) 
\|f\|_2^2 
\end{align*}
hold. 
\end{lemma}

\begin{proof}
Note that 
\begin{align*}
& \langle f_+, (\pi \cdot f)_+\rangle \\
& = \sum_{x\in \calI_{\pi, f>0} } f(x) (\pi \cdot f)(x) \\
& = \sum_{x\in \calI_{\pi, f>0} } f(x) f(\pi ^\mo x) \\
& = \sum_{x\in \calI_{\pi, f>0} } (f_\str(x) + f_\sml(x))  (f_\str(\pi ^\mo x)  + f_\sml(\pi ^\mo x) )\\
& = \sum_{x\in \calI_{\pi, f>0} } 
(f_\str(x)f_\str(\pi ^\mo x)  + f_\str(x)f_\sml(\pi ^\mo x) +
f_\sml(x)f_\str(\pi ^\mo x)  + f_\sml(x)f_\sml(\pi ^\mo x) )\\
& = \frac{|\calI_{\pi, f>0} |}n \|f_\str\|_2^2 +
\sum_{x\in \calI_{\pi, f>0} } 
(f_\str(x)f_\sml(\pi ^\mo x) +
f_\sml(x)f_\str(\pi ^\mo x)  + f_\sml(x)f_\sml(\pi ^\mo x) ),
\end{align*}
which shows that 
\begin{align*}
& 
\left|
\langle f_+, (\pi \cdot f)_+\rangle - \frac{|\calI_{\pi, f>0} |}n \|f_\str\|_2^2 
\right|\\
& \leq 
\|f_\str 1_{\calI_{\pi, f>0}}  \|_2 
\|f_\sml 1_{\pi ^\mo (\calI_{\pi, f>0} )} \|_2 
+ 
\|f_\sml 1_{\calI_{\pi, f>0}}  \|_2 
\|f_\str 1_{\pi ^\mo (\calI_{\pi, f>0} )} \|_2 
+
\|f_\sml 1_{\pi ^\mo (\calI_{\pi, f>0} )} \|_2 
\|f_\sml 1_{\calI_{\pi, f>0}}  \|_2 
\\
& \leq  
\|f_\str \|_2  \frac{\sqrt{|\calI_{\pi, f>0} |}}{\sqrt n}
\|f_\sml  \|_2 
+ 
\|f_\sml \|_2 
\|f_\str \|_2  \frac{\sqrt{|\pi ^\mo(\calI_{\pi, f>0} )|}}{\sqrt n}
+
\|f_\sml \|_2^2
\\
& =
2\frac{\sqrt{|\calI_{\pi, f>0} |}}{\sqrt n}
\|f_\sml  \|_2 \|f_\str \|_2  
+
\|f_\sml \|_2^2
\\
& \leq 
\left(
2\frac{\sqrt{|\calI_{\pi, f>0} |}}{\sqrt n}
\sqrt \varkappa 
+ 
\varkappa 
\right) 
\|f\|_2^2 
\quad
\text{(using \cref{Lemma:DecompositionStrSml})}.
\end{align*}
\end{proof}

For a permutation $\pi: V \to V$ and a subset $\scrV$ of $V$, define 
$$\Sigma_{\pi, \scrV }
:= 
\frac{\langle \pi  1_{\scrV }, 1_{\scrV } \rangle + \langle \pi  1_{\scrV ^c}, 1_{\scrV ^c} \rangle}{ n} .$$

\begin{lemma}
\label{Lemma:SumA11A22}
For any permutation $\pi: V \to V$ and for subsets $\calA,\calB$ of $V$, the bound 
\begin{align*}
|\Sigma_{\pi, \calA} - \Sigma_{\pi, \calB}|
& \leq 
\sqrt{2}\sqrt{ 
\frac{|\calA\cap \calB^c| + |\calA^c\cap \calB|}n
} .
\end{align*}
holds. 
\end{lemma}

\begin{proof}
For any two subsets $A, B$ of $V$, note that 
$$1_A - 1_B  = 1_{A \cap B^c} - 1_{B \cap A^c}$$
holds, which shows that 
\begin{align*}
1_{\calA} & =  1_\calB+   1_{\calA}1_{\calB^c} - 1_{\calA^c}1_\calB ,\\
1_{\calA^c}  & =  1_{\calB^c} +  1_{\calA^c}1_{\calB} - 1_{\calA}1_{\calB^c} .
\end{align*}
Using this, we obtain 
\begin{align*}
& \langle \pi  1_{\calA}, 1_{\calA} \rangle + \langle \pi  1_{\calA^c}, 1_{\calA^c} \rangle\\
& = \langle \pi  (1_\calB+   1_{\calA}1_{\calB^c} - 1_{\calA^c}1_\calB ), 1_{\calA} \rangle  \\
& + \langle \pi  (1_{\calB^c} +  1_{\calA^c}1_{\calB} - 1_{\calA}1_{\calB^c}), 1_{\calA^c} \rangle\\
& = 
\langle \pi  1_\calB, 1_{\calA} \rangle  
+\langle \pi    1_{\calA}1_{\calB^c} , 1_{\calA} \rangle  
- \langle \pi   1_{\calA^c}1_\calB , 1_{\calA} \rangle  \\
& 
+\langle \pi  1_{\calB^c} , 1_{\calA^c} \rangle
+\langle \pi  1_{\calA^c}1_{\calB} , 1_{\calA^c} \rangle
-\langle \pi 1_{\calA}1_{\calB^c}, 1_{\calA^c} \rangle \\
& = 
\langle \pi  1_\calB, 1_\calB+   1_{\calA}1_{\calB^c} - 1_{\calA^c}1_\calB \rangle  
+\langle \pi    1_{\calA}1_{\calB^c} , 1_{\calA} \rangle  
- \langle \pi   1_{\calA^c}1_\calB , 1_{\calA} \rangle  \\
& 
+\langle \pi  1_{\calB^c} , 1_{\calB^c} +  1_{\calA^c}1_{\calB} - 1_{\calA}1_{\calB^c}  \rangle
+\langle \pi  1_{\calA^c}1_{\calB} , 1_{\calA^c} \rangle
-\langle \pi 1_{\calA}1_{\calB^c}, 1_{\calA^c} \rangle \\
& = 
\langle \pi  1_\calB, 1_\calB \rangle  
+\langle \pi  1_\calB,  1_{\calA}1_{\calB^c}  \rangle  
-\langle \pi  1_\calB, 1_{\calA^c}1_\calB \rangle  
+\langle \pi    1_{\calA}1_{\calB^c} , 1_{\calA} \rangle  
- \langle \pi 1_{\calA^c}1_\calB , 1_{\calA} \rangle  \\
& 
+
\langle \pi  1_{\calB^c} , 1_{\calB^c} \rangle
+\langle \pi  1_{\calB^c} , 1_{\calA^c}1_{\calB}\rangle
-\langle \pi  1_{\calB^c} , 1_{\calA}1_{\calB^c}  \rangle
+\langle \pi  1_{\calA^c}1_{\calB} , 1_{\calA^c} \rangle
-\langle \pi 1_{\calA}1_{\calB^c}, 1_{\calA^c} \rangle \\
& = 
\langle \pi  1_\calB, 1_\calB \rangle  
+\langle \pi  1_{\calB^c} , 1_{\calB^c} \rangle
\\
&+\langle \pi  1_\calB,  1_{\calA}1_{\calB^c}  \rangle  -\langle \pi  1_{\calB^c} , 1_{\calA}1_{\calB^c}  \rangle
+\langle \pi    1_{\calA}1_{\calB^c} , 1_{\calA} \rangle  -\langle \pi 1_{\calA}1_{\calB^c}, 1_{\calA^c} \rangle 
 \\
& 
+\langle \pi  1_{\calB^c} , 1_{\calA^c}1_{\calB}\rangle -\langle \pi  1_\calB, 1_{\calA^c}1_\calB \rangle  
+\langle \pi  1_{\calA^c}1_{\calB} , 1_{\calA^c} \rangle - \langle \pi  1_{\calA^c}1_\calB , 1_{\calA} \rangle 
  \\
& = 
\langle \pi  1_\calB, 1_\calB \rangle  
+\langle \pi  1_{\calB^c} , 1_{\calB^c} \rangle
\\
&+\langle \pi  (1_\calB-1_{\calB^c} ),  1_{\calA}1_{\calB^c}  \rangle  
+\langle \pi    1_{\calA}1_{\calB^c} , 1_{\calA} -1_{\calA^c} \rangle 
-\langle \pi (1_\calB - 1_{\calB^c} ), 1_{\calA^c}1_\calB \rangle  
-\langle \pi  1_{\calA^c}1_{\calB} , 1_{\calA} -1_{\calA^c}  \rangle 
  \\
& = 
\langle \pi  1_\calB, 1_\calB \rangle  
+\langle \pi  1_{\calB^c} , 1_{\calB^c} \rangle
\\
&+\langle \pi  (1_\calB-1_{\calB^c} ),  1_{\calA}1_{\calB^c}  \rangle  
-\langle \pi (1_\calB - 1_{\calB^c} ), 1_{\calA^c}1_\calB \rangle  
+\langle \pi    1_{\calA}1_{\calB^c} , 1_{\calA} -1_{\calA^c} \rangle
-\langle \pi  1_{\calA^c}1_{\calB} , 1_{\calA} -1_{\calA^c}  \rangle 
  \\
& = 
\langle \pi  1_\calB, 1_\calB \rangle  
+\langle \pi  1_{\calB^c} , 1_{\calB^c} \rangle
\\
&+\langle \pi  (1_\calB-1_{\calB^c} ),  1_{\calA}1_{\calB^c}  -1_{\calA^c}1_\calB \rangle  
+\langle \pi    (1_{\calA}1_{\calB^c} -1_{\calA^c}1_{\calB} ), 1_{\calA} -1_{\calA^c}  \rangle 
  \\
& = 
\langle \pi  1_\calB, 1_\calB \rangle  
+\langle \pi  1_{\calB^c} , 1_{\calB^c} \rangle
\\
&+\langle \pi  (1_\calB-1_{\calB^c} +1_{\calA} -1_{\calA^c}),  1_{\calA}1_{\calB^c}  -1_{\calA^c}1_\calB \rangle  
\end{align*}
hold, which implies that 
\begin{align*}
& \langle \pi  1_{\calA}, 1_{\calA} \rangle + \langle \pi  1_{\calA^c}, 1_{\calA^c} \rangle
- 
\langle \pi  1_\calB, 1_\calB \rangle  
-\langle \pi  1_{\calB^c} , 1_{\calB^c} \rangle
\\
&=\langle \pi  (1_\calB-1_{\calB^c} +1_{\calA} -1_{\calA^c}),  1_{\calA}1_{\calB^c}  -1_{\calA^c}1_\calB \rangle  .
\end{align*}
Noting that 
\begin{align*}
& |\langle \pi  (1_\calB-1_{\calB^c} +1_{\calA} -1_{\calA^c}),  1_{\calA}1_{\calB^c}  -1_{\calA^c}1_\calB \rangle  | \\
& \leq \|1_\calB-1_{\calB^c} +1_{\calA} -1_{\calA^c}\|_2 \| 1_{\calA}1_{\calB^c}  -1_{\calA^c}1_\calB   \|_2\\
& = \sqrt{2n}\sqrt{ |\calA\cap \calB^c| + |\calA^c\cap \calB|},
\end{align*}
the result follows. 
\end{proof}

\section{Dichotomy yields an index two subgroup}

\begin{condition}
\label{Condition22}

Let $\calG$ be a group acting transitively on $V$, and assume that no index two subgroup of $\calG$ acts transitively on $V$. 
\end{condition}

Henceforth, we assume that \cref{Condition22} holds.

\begin{proposition}
[Dichotomy provides an index two subgroup]
\label{Lemma:SubgrpOfIndex2}
Suppose $\varkappa$ lies in $[0,1/260]$ and $\xi \leq \frac 45$. 
Assume that 
$$
\langle f_+, (\tau \cdot f)_+\rangle 
\in 
(\delta \| f_+\|_2^2 , (1-\delta)  \| f_+\|_2^2)
$$
holds for no $\tau \in \calG$. 
Then 
$$H := \{\tau \in \calG \,|\, 
\langle f_+, (\tau \cdot f)_+\rangle 
\geq 
(1-\delta) \| f_+\|_2^2
\}$$
is a subgroup of $\calG$ of index two. 
\end{proposition}

\begin{proof}
Note that 
\begin{align*}
\langle f_+, (\tau \cdot f)_+\rangle 
& = \langle f1_{f>0}, (\tau \cdot f) 1_{\tau \cdot f > 0}\rangle \\
& = \sum_{x\in \supp(f_+) \cap \tau \cdot \supp(f_+)} 
f(x)  (\tau \cdot f)(x) .
\end{align*}
First, we show that $H \neq \calG$. 
For any $\tau \in \calG$, consider the map 
$$\supp(f_+) \cap \tau \cdot \supp(f_+) \to \supp(f_+) \times \supp(f_+) , 
\quad 
x \mapsto (x, \tau^\mo \cdot x)
.$$
This induces a $t:1$ map 
$$
\coprod_{\tau \in \calG } \supp(f_+) \cap \tau \cdot \supp(f_+) \to \supp(f_+) \times \supp(f_+) 
,$$
where $t$ denotes the size of the stablizer of any point of $V$ under the transitive action of $\calG$. This shows that 
$$
\frac 1t
\sum_{\tau\in \calG}
\sum_{x\in \supp(f_+) \cap \tau \cdot \supp(f_+)}
f(x) f(\tau^\mo x)
= 
\sum_{x, y \in \supp(f_+) } f(x) f(y) 
,$$
which yields 
$$
\frac 1t
\sum_{\tau\in \calG}
\langle f_+, (\tau \cdot f)_+\rangle 
= 
\left(\sum_{x \in \supp(f_+) } f(x) \right)^2 
.$$
If $H = \calG$, then 
\begin{align*}
(1-\delta)  \| f_+\|_2^2|\calG|
& \leq 
\sum_{\tau\in \calG}
\sum_{x\in \supp(f_+) \cap \tau \cdot \supp(f_+)}
f(x) f(\tau^\mo x) \\
& = t \langle f_+ , 1_{f>0}\rangle^2 \\
& \leq t\|f_+\|_2^2 |\supp(f_+)|\\
&\leq t \|f_+\|_2^2 \frac {1+ \varsigma} 2 |V|
\quad 
\text{(using \cref{Lemma:DiffOfL2NormOFf+f-})}\\
&\leq  \|f_+\|_2^2 \frac {1+ \varsigma} 2 |\calG|,
\end{align*}
which implies 
$$
\delta 
\geq \frac {1- \varsigma} 2.$$
Since 
$$
\delta < 
\frac 12 
\left(
\frac {3\left(1 - \varkappa\right)} {1  + \varsigma + \varkappa + 2 \sqrt 2 \sqrt{\varkappa}}
\left(
\frac 1{\sqrt 2}
-
\varsigma 
\right) ^2 
-1
\right)
< \frac{1-\varsigma}2
\quad 
\text{ for } \varkappa \in (0, \tfrac 7{10}) 
$$
and $\varkappa \leq \frac 1{260}$, we conclude that $H \neq \calG$. 
Now, we prove that $H$ forms a subgroup of $\calG$. Note that $H$ contains the identity element of $\calG$. 
Let $\tau_1, \tau_2$ be elements of $H$. 
Note that 
\begin{align*}
& \langle f_+ - (\tau_1 \cdot f)_+,   (\tau_1 \tau_2 \cdot f)_+ -  (\tau_1 \cdot f)_+ \rangle \\
& = \langle f_+ ,   (\tau_1 \tau_2 \cdot f)_+ -  (\tau_1 \cdot f)_+ \rangle 
- \langle (\tau_1 \cdot f)_+,   (\tau_1 \tau_2 \cdot f)_+ -  (\tau_1 \cdot f)_+ \rangle \\
& = \langle f_+ ,   (\tau_1 \tau_2 \cdot f)_+ \rangle 
- \langle f_+ ,  (\tau_1 \cdot f)_+ \rangle 
- \langle (\tau_1 \cdot f)_+,   (\tau_1 \tau_2 \cdot f)_+  \rangle
+\langle (\tau_1 \cdot f)_+,    (\tau_1 \cdot f)_+ \rangle \\
& = \langle f_+ ,   (\tau_1 \tau_2 \cdot f)_+ \rangle 
- \langle f_+ ,  (\tau_1 \cdot f)_+ \rangle 
- \langle f_+ ,  (\tau_2 \cdot f)_+ \rangle 
+\| f_+\|_2^2
\end{align*}
hold. Also note that 
\begin{align*}
& |\langle f_+ - (\tau_1 \cdot f)_+,   (\tau_1 \tau_2 \cdot f)_+ -  (\tau_1 \cdot f)_+ \rangle |^2\\
& \leq \| f_+ - (\tau_1 \cdot f)_+\|_2^2 \|  (\tau_1 \tau_2 \cdot f)_+ -  (\tau_1 \cdot f)_+ \|_2^2 \\
& =  
(\|f_+\|_2^2 + \|(\tau_1 \cdot f)_+\|_2^2 
- 2\langle f_+ ,  (\tau_1 \cdot f)_+\|_2^2 \rangle ) \\
& \qquad (\|  (\tau_1 \tau_2 \cdot f)_+ \|_2^2 + |(\tau_1 \cdot f)_+ \|_2^2
- 2 \langle (\tau_1 \tau_2 \cdot f)_+ , (\tau_1 \cdot f)_+ \rangle   )\\
& 
= 
4
(\|f_+\|_2^2 - \langle f_+, (\tau_1\cdot f)_+\rangle ) 
(\|f_+\|_2^2 - \langle f_+, (\tau_2 \cdot f)_+\rangle ) 
\\
& \leq 4 \delta^2 \|f_+\|_2^4.
\end{align*}
This shows that 
\begin{align*}
& \langle f_+, (\tau_1 \tau_2 \cdot f)_+\rangle \\
& = 
 \langle f_+, (\tau_1  \cdot f)_+\rangle 
+ \langle f_+, ( \tau_2 \cdot f)_+\rangle 
- \|f_+\|_2^2 +
\langle f_+ - (\tau_1 \cdot f)_+,   (\tau_1 \tau_2 \cdot f)_+ -  (\tau_1 \cdot f)_+ \rangle\\
& \geq 
( 1- \delta + 1- \delta -1 - 2\delta)  \|f_+\|_2^2 \\
& = (1 -4\delta)  \|f_+\|_2^2 \\
& > \delta  \|f_+\|_2^2 
\quad 
\text{(since $\xi \leq \frac 45$ and $\varkappa >0$)} .
\end{align*}
Since 
$$
\langle f_+, (\tau \cdot f)_+\rangle 
\in 
(\delta \| f_+\|_2^2, (1-\delta) \| f_+\|_2^2 )
$$
holds for no $\tau \in \calG$, 
it follows that 
$$
\langle f_+, (\tau_1 \tau_2 \cdot f)_+\rangle
\geq 
(1- \delta)  \|f_+\|_2^2 .$$
Hence, $H$ is a proper subgroup of $\calG$. We claim that $H$ has index two in $\calG$. 
Note that 
\begin{align*}
& \left|\sum_{x \in \supp(f_+) } f(x) \right| \\
& \geq \|f_\str\|_2 \frac{|\supp(f_+) |}{\sqrt n} 
-
\varsigma \|f_\str\|_2 \sqrt{|\supp(f_+) |} 
\quad 
\text{(using \cref{Lemma:fsmlL1bdd})}\\
& =
\|f_\str\|_2 \sqrt{|\supp(f_+) |}
\left(
\frac{\sqrt{|\supp(f_+) |}}{\sqrt n} 
-
\varsigma
\right)
\\
& \geq 
\|f_\str\|_2 
\sqrt{|\supp(f_+) |}
\left(
\frac 1{\sqrt 2} 
-
\varsigma 
\right) \\
& \geq 
\left(\frac {2(1 - \varkappa)} {1  + \varsigma + \varkappa + 2 \sqrt 2 \sqrt{\varkappa}}\right)^{1/2}
  \|f_+\|_2 
  \frac {\sqrt {|V|}}{\sqrt 2}
\left(
\frac 1{\sqrt 2}
-
\varsigma 
\right) 
\quad 
\text{(using \cref{Lemma:DecompositionStrSml}, \cref{Lemma:DiffOfL2NormOFf+f-} and $\varsigma \leq \tfrac 1 {\sqrt 2}$)}\\
\end{align*}
Using $\varsigma \leq \frac 1 {\sqrt 2}$, it follows that 
\begin{align*}
\frac {2(1 - \varkappa)} {1  + \varsigma + \varkappa + 2 \sqrt 2 \sqrt{\varkappa}}
  \|f_+\|_2 ^2 
\frac {|\calG|}2
\left(
\frac 1{\sqrt 2}
-
\varsigma 
\right) ^2
& \leq t^2 \left(\sum_{x \in \supp(f_+) } f(x) \right)^2 \\
& = \sum_{\tau\in \calG} \langle f_+, (\tau \cdot f)_+\rangle  \\
& = 
\sum_{\tau\in H } \langle f_+, (\tau \cdot f)_+\rangle 
+ \sum_{\tau\in \calG \setminus H } \langle f_+, (\tau \cdot f)_+\rangle \\
& \leq 
\sum_{\tau\in H } \langle f_+, (\tau \cdot f)_+\rangle 
+ \delta (|\calG| - |H|) \|f_+\|_2^2 \\
& \leq 
\sum_{\tau\in H } \|f_+\|_2 \|(\tau \cdot f)_+\|_2 
+ \delta (|\calG| - |H|) \|f_+\|_2^2 \\
& = \sum_{\tau\in H } \|f_+\|_2^2 + \delta (|\calG| - |H|) \|f_+\|_2^2 \\
& = |H| \|f_+\|_2^2 + \delta (|\calG| - |H|) \|f_+\|_2^2 ,
\end{align*}
which gives 
\begin{align*}
\frac {1 - \varkappa} {1  + \varsigma + \varkappa + 2 \sqrt 2 \sqrt{\varkappa}}
\left(
\frac 1{\sqrt 2}
-
\varsigma 
\right) ^2 
& \leq (1- \delta) \frac{|H|}{|\calG|}  + \delta ,
\end{align*}
and hence 
$$
\frac{|H|}{|\calG|}  
\geq 
\frac {
\frac {1 - \varkappa} {1  + \varsigma + \varkappa + 2 \sqrt 2 \sqrt{\varkappa}}
\left(
\frac 1{\sqrt 2}
-
\varsigma 
\right) ^2 
- \delta }{1 - \delta }.
$$
Using 
$$
\delta 
< 
\frac 12 
\left(
\frac {3\left(1 - \varkappa\right)} {1  + \varsigma + \varkappa + 2 \sqrt 2 \sqrt{\varkappa}}
\left(
\frac 1{\sqrt 2}
-
\varsigma 
\right) ^2 
-1
\right),$$
we obtain 
$\frac{|H|}{|\calG|}  > \frac 13.$
Since $H\neq \calG$, it follows that $H$ has index two in $\calG$. 
This completes the proof of the lemma. 
\end{proof}

\begin{lemma}
\label{Lemma:Suppf+EqualsCalO}
Suppose $\varkappa$ lies in $[0,1/260]$ and $\xi \leq \frac 45$. 
Assume that 
$$
\langle f_+, (\tau \cdot f)_+\rangle 
\in 
(\delta \| f_+\|_2^2 , (1-\delta)  \| f_+\|_2^2)
$$
holds for no $\tau \in \calG$. 
Then for some orbit $\calO$ of the action of 
$$H = \{\tau \in \calG \,|\, 
\langle f_+, (\tau \cdot f)_+\rangle 
\geq 
(1-\delta) \| f_+\|_2^2
\}$$
on $V$, 
the bound 
\begin{align}
\label{Eqn:Suppf+EqualsCalO}
\frac{|\supp(f_+) \cap \calO^c|}{n}
& \leq \frac \varsigma {\sqrt 2} 
+ \frac {\sqrt \delta} 2 
\sqrt{\frac {1  + \varsigma + \varkappa + 2 \sqrt 2 \sqrt{\varkappa}} {2(1 - \varkappa)} }
\end{align}
holds. 
\end{lemma}

\begin{proof}
By \cref{Lemma:SubgrpOfIndex2}, $H$ is a subgroup of $\calG$ of index two. 
By \cref{Condition22}, under the action of $H$, the set $V$ decomposes into two distinct orbits, denoted by $\calO_1, \calO_2$. 
For any subsets $A, B$ of $V$ contained in $\calO_1, \calO_2$ (in some order) and an element $\tau \in \calG \setminus H$, 
consider the map 
$$A \cap \tau \cdot B
\to A \times B, 
\quad 
x\mapsto (x, \tau^\mo \cdot x)
.
$$
This induces a $t:1$ map 
$$
\coprod_{\tau \in H^c}
A \cap \tau B
\to 
A \times B
, \quad 
x = x_\tau \mapsto (x, \tau^\mo \cdot x) ,
$$
where $t$ denotes the size of the stablizer of any point of $V$ under the transitive action of $\calG$. 
This shows that 
$$
\sum_{a\in A, b\in B} f(a) f(b) 
= \frac 1t
\sum_{\tau \in H^c }\sum_{x\in A \cap \tau  B}f(x) (\tau \cdot f) (x) 
.$$
We obtain 
$$
\sum_{a\in \supp(f_+) \cap \calO_1, b\in \supp(f_+) \cap \calO_2} f(a) f(b) 
= \frac 1t
\sum_{\tau \in H^c }\sum_{x\in (\supp(f_+) \cap \calO_1)  \cap \tau  (\supp(f_+) \cap \calO_2)}f(x) (\tau \cdot f) (x) 
,
$$
$$
\sum_{a\in \supp(f_+) \cap \calO_2, b\in \supp(f_+) \cap \calO_1} f(a) f(b) 
= \frac 1t
\sum_{\tau \in H^c }\sum_{x\in (\supp(f_+) \cap \calO_2)  \cap \tau  (\supp(f_+) \cap \calO_1)}f(x) (\tau \cdot f) (x) 
.$$
Combining the above yields 
\begin{align*}
& 2\sum_{a\in \supp(f_+) \cap \calO_1, b\in \supp(f_+) \cap \calO_2} f(a) f(b) \\
& = \frac 1t
\sum_{\tau \in H^c }
\left(
\sum_{x\in (\supp(f_+) \cap \calO_1)  \cap \tau  (\supp(f_+) \cap \calO_2)}
+
\sum_{x\in (\supp(f_+) \cap \calO_2)  \cap \tau  (\supp(f_+) \cap \calO_1)}f(x) (\tau \cdot f) (x) 
\right) \\
& = \frac 1t 
\sum_{\tau \in H^c }\sum_{x\in \supp(f_+) \cap \tau \supp(f_+)}f(x) (\tau \cdot f) (x) \\
& = \frac 1t \sum_{\tau \in H^c } \langle f_+, (\tau \cdot f)_+\rangle \\
& \leq \frac \delta t |H|  \|f_+\|_2^2.
\end{align*}
It follows that 
\begin{align*}
\langle f_+ ,1_{\calO_1}\rangle 
\langle f_+ ,1_{\calO_2}  \rangle 
& = \langle f_+ 1_{\calO_1}, 1_V \rangle 
\langle f_+ 1_{\calO_2}, 1_V \rangle \\
& = 
\sum_{a\in \supp(f_+) \cap \calO_1} f(a) 
\times 
\sum_{b\in \supp(f_+) \cap \calO_2} f(b) \\
& = 
\sum_{a\in \supp(f_+) \cap \calO_1, b\in \supp(f_+) \cap \calO_2} f(a) f(b) \\
& \leq 
\frac 14\delta |V|  \|f_+\|_2^2.
\end{align*}
Hence, for some orbit $\calO$ of the action of $H$ on $V$, 
$$
\langle f_+ ,1_{\calO^c}\rangle 
\leq 
\sqrt{\frac \delta 4 |V| }  \|f_+\|_2
.$$
Note that 
\begin{align*}
& \|f_\str\|_2\frac{|\supp(f_+) \cap \calO^c|}{\sqrt n} 
-
\varsigma \|f_\str\|_2 \sqrt{|\supp(f_+) \cap \calO^c|} \\
& \leq \left|\sum_{x \in \supp(f_+) \cap \calO^c} f(x) \right| 
\quad 
\text{(using \cref{Lemma:fsmlL1bdd})}\\
& = \langle f_+ ,1_{\calO^c}\rangle \\
& \leq \sqrt{\frac \delta 4 |V| }  \|f_+\|_2 \\
& \leq \sqrt{\frac \delta 4 |V| }  
\sqrt{\frac {1  + \varsigma + \varkappa + 2 \sqrt 2 \sqrt{\varkappa}} 2 }\|f\|_2 \\
& \leq \sqrt{\frac \delta 4 |V| }  
\sqrt{\frac {1  + \varsigma + \varkappa + 2 \sqrt 2 \sqrt{\varkappa}} 2 }
\frac 1{\sqrt{1- \varkappa}}\|f_\str\|_2 ,
\end{align*}
which implies that 
\begin{align*}
\frac{|\supp(f_+) \cap \calO^c|}{n}
& \leq 
\varsigma  
\frac{ \sqrt{|\supp(f_+) \cap \calO^c|} }{\sqrt n}
+ \frac {\sqrt \delta} 2 
\sqrt{\frac {1  + \varsigma + \varkappa + 2 \sqrt 2 \sqrt{\varkappa}} {2(1 - \varkappa)} }
\\
& \leq \frac \varsigma {\sqrt 2} 
+ \frac {\sqrt \delta} 2 
\sqrt{\frac {1  + \varsigma + \varkappa + 2 \sqrt 2 \sqrt{\varkappa}} {2(1 - \varkappa)} }.
\end{align*}
\end{proof}

\section{Dichotomy yields lower spectral gap}

Henceforth, we assume that \cref{ConditionConn} holds. 

\begin{condition}
\label{ConditionConn}
\quad 
\begin{enumerate}
\item  Let $\rho_1, \ldots, \rho_d: V\to V$ be permutations such that 
$T = P_{\rho_1} + \cdots + P_{\rho_d}$.
\item 
If $\calG$ admits an index two subgroup, then let $\nu$ denote the largest element of $[0, 1]$ such that for any subgroup $H$ of index two in $\calG$, 
there exists an integer $1\leq i_H \leq d$ such that 
$$
\Sigma_{\rho_{i_H}, \scrO} \geq \nu$$
holds for any orbit $\scrO$ of the action of $H$ on $V$. 
\end{enumerate}
\end{condition}

\begin{proposition}
[Dichotomy yields lower spectral gap]
\label{Lemma:DichotomyYieldsLowerBdd}
Suppose $\varkappa$ lies in $[0,1/260]$ and 
$$
\xi 
\leq 
\begin{cases}
\frac 45 & \text{ if } \nu = 1, \\
\frac {123}{1000} & \text{ if } \nu \geq 1/2.
\end{cases}
$$
If $$
\langle f_+, (\tau \cdot f)_+\rangle 
\in 
(\delta \langle f_+, f_+\rangle , (1-\delta) \langle f_+, f_+\rangle )
$$
holds for no $\tau \in \calG$, 
then the bounds 
$$
\varkappa \geq \frac \beta 2 \geq
\begin{cases}
\frac 1{10d}
& \text{ if } \nu = 1, \\
\frac {1}{50000d} & \text{ if } \nu \geq 1/2
\end{cases}
$$
hold.
\end{proposition}

\begin{proof}
By \cref{Lemma:SubgrpOfIndex2}, 
$$H = \{\tau \in \calG \,|\, 
\langle f_+, (\tau \cdot f)_+\rangle 
\geq 
(1-\delta) \| f_+\|_2^2
\}$$
forms a subgroup of $\calG$ of index two. 
By \cref{Lemma:Suppf+EqualsCalO}, for some orbit $\calO$ of the action of $H$ on $V$, 
the bound 
\begin{align*}
\frac{|\supp(f_+) \cap \calO^c|}{n}
& \leq \frac \varsigma {\sqrt 2} 
+ \frac {\sqrt \delta} 2 
\sqrt{\frac {1  + \varsigma + \varkappa + 2 \sqrt 2 \sqrt{\varkappa}} {2(1 - \varkappa)} }
\end{align*}
holds. 
By \cref{ConditionConn}(2), it follows that the set $\rho_i(\calO) \cap \calO$ contains at least $\nu |\calO|$ elements for some $1\leq i \leq d$. 
For $\pi = \rho_i$, note that 
\begin{align*}
& |\Sigma_{\pi, f>0} - \Sigma_{\pi, \calO}|\\
& =\frac 1n  | \langle \pi  1_{f>0}, 1_{f>0} \rangle + \langle \pi  1_{f<0}, 1_{f<0} \rangle 
- 
\langle \pi  1_\calO, 1_\calO \rangle  
-\langle \pi  1_{\calO^c} , 1_{\calO^c} \rangle
| \\
& \leq \frac {\sqrt 2}{\sqrt n}\sqrt{ |\supp(f_+)\cap \calO^c| + |\supp(f_-)\cap \calO|} 
\quad 
\text{(using \cref{Lemma:SumA11A22})}\\
& = \frac {\sqrt 2}{\sqrt n}\sqrt{ |\supp(f_+)\cap \calO^c| + |\calO| -  |\supp(f_+) \cap \calO|} \\
& = \frac {\sqrt 2}{\sqrt n}\sqrt{ |\supp(f_+) \cap \calO^c| + |\calO| -  |\supp(f_+) | + |\supp(f_+)\cap \calO^c|} \\
& \leq \frac {\sqrt 2}{\sqrt n}\sqrt{ 2|\supp(f_+) \cap \calO^c| } \\
& \leq 2
\left(
\frac \varsigma {\sqrt 2} 
+ \frac {\sqrt \delta} 2 
\sqrt{\frac {1  + \varsigma + \varkappa + 2 \sqrt 2 \sqrt{\varkappa}} {2(1 - \varkappa)} }
\right)^{1/2}
\end{align*}
holds, which gives 
\begin{align*}
d\beta
& = \frac 1n 
(\langle T1_{f > 0}, 1_{f > 0} \rangle + \langle T1_{f < 0}, 1_{f < 0} \rangle ) \\
& \geq \Sigma_{\pi, f>0}\\
& \geq \Sigma_{\pi, \calO}
-  2
\left(
\frac \varsigma {\sqrt 2} 
+ \frac {\sqrt \delta} 2 
\sqrt{\frac {1  + \varsigma + \varkappa + 2 \sqrt 2 \sqrt{\varkappa}} {2(1 - \varkappa)} }
\right)^{1/2}
\\
& \geq \nu 
-  2 
\left(
\frac \varsigma {\sqrt 2} 
+ \frac {\sqrt \delta} 2 
\sqrt{\frac {1  + \varsigma + \varkappa + 2 \sqrt 2 \sqrt{\varkappa}} {2(1 - \varkappa)} }
\right)^{1/2}.
\end{align*}
If $\nu =1$ and $\varkappa \leq \frac 1{260}$, then for $\xi \leq \frac 45$, we have 
$$
d\beta
\geq \frac 15 .
$$
If $\nu \geq \frac 12$ and $\varkappa\leq \frac 1{54550}$, then for $\xi \leq \frac{123}{1000}$, we have 
$$
d \beta 
\geq \frac {1}{25000}.
$$
Using \cref{Lemma:LowerBddForKappaGeneral}, 
the result follows.
\end{proof}

\section{Lack of dichotomy also yields lower spectral gap}

Consider the following condition. 
\begin{condition}
\label{Left2Right}
For any $\tau \in \calG$, there exists a permutation $\varrho_\tau: V \to V$ such that for any subsets $A, B$ of $V$, the number of edges between $A, B$ and those between $\tau(A), \varrho_\tau(B)$ are equal, i.e., 
$\langle T1_A, 1_B \rangle  = \langle T1_{\tau ( A)}, 1_{\varrho_\tau(B)} \rangle$ holds. 
\end{condition}

\begin{proposition}
\label{Lemma:LackOfDichotomyGivesLowerSpecGap}
Assume that \cref{Left2Right} holds. 
If $$
\langle f_+, (\tau \cdot f)_+\rangle 
\in 
(\delta \langle f_+, f_+\rangle , (1-\delta) \langle f_+, f_+\rangle )
$$
holds for some $\tau \in \calG$, 
then 
$$1+ \mu 
\geq 
\begin{cases}
\frac {1- \mu_2}{2521}
& \text{ if } \nu = 1, \\
\frac {1- \mu_2}{54632} & \text{ if } \nu \geq \frac 12. 
\end{cases}
$$
\end{proposition}

\begin{proof}
Assume that $\varkappa \leq \frac 1{260}$. Take 
$$
\xi 
= 
\begin{cases}
\frac 45 & \text{ if } \nu = 1, \\
\frac {123}{1000} & \text{ if } \nu \geq 1/2.
\end{cases}
$$
Write 
$X = \supp(f_+) \cap \tau(\supp(f_+))$. 
Note that 
\begin{align*}
\frac{|X|}n \|f\|_2^2 
& \geq 
\langle f_+, (\tau\cdot f)_+\rangle
- 
\left(  \frac{|X|}n \varkappa+ 
2\frac{\sqrt{|X|}}{\sqrt n}
\sqrt \varkappa 
+ 
\varkappa 
\right) 
\|f\|_2^2 
\quad 
\text{(by \cref{Lemma:InnerProdOfPlusPartWithPlusPartOfATranslate})} \\
& >  
\delta \|f_+\|_2^2
- 
\left(  \frac{|X|}n \varkappa+ 
2\frac{\sqrt{|X|}}{\sqrt n}
\sqrt \varkappa 
+ 
\varkappa 
\right) 
\|f\|_2^2  \\
& \geq 
\delta \vartheta^2 \|f\|_2^2
- 
\left(  \frac{|X|}n \varkappa+ 
2\frac{\sqrt{|X|}}{\sqrt n}
\sqrt \varkappa 
+ 
\varkappa 
\right) 
\|f\|_2^2 
\quad 
\text{(by \cref{Lemma:ControllingL2NormFromBelow})},
\end{align*}
which implies 
\begin{align*}
\frac{|X|}n
& \geq 
\delta \vartheta^2 
- 
\left(  \frac{|X|}n \varkappa+ 
2\frac{\sqrt{|X|}}{\sqrt n}
\sqrt \varkappa 
+ 
\varkappa 
\right) .
\end{align*}
This shows that 
\begin{align*}
( 1 + \varkappa) \frac{|X|}n 
+
2\frac{\sqrt{|X|}}{\sqrt n}
\sqrt \varkappa 
- 
(\delta  \vartheta^2 - \varkappa) 
\geq 0
\end{align*}
Noting that $\varkappa + (1+ \varkappa) (\delta \vartheta^2 - \varkappa) \geq 0$ if $\varkappa \in [0,1/260]$, we obtain 
\begin{align*}
 \frac{\sqrt{|X|}}{\sqrt n }
& \geq 
\frac {- \sqrt \varkappa + \sqrt{\varkappa + (1+ \varkappa) (\delta \vartheta^2 - \varkappa)}}{1+ \varkappa} \\
& = 
\frac {(1+ \varkappa) (\delta \vartheta^2 - \varkappa)}{( \sqrt \varkappa + \sqrt{\varkappa + (1+ \varkappa) (\delta \vartheta^2 - \varkappa)})(1+ \varkappa)} \\
& = 
\frac {\delta \vartheta^2 - \varkappa}{\sqrt \varkappa + \sqrt{\varkappa + (1+ \varkappa) (\delta \vartheta^2 - \varkappa)}} .
\end{align*}
Moreover, we have
\begin{align*}
\frac{|X|}n \|f\|_2^2 
& \leq 
\langle f_+, (\tau\cdot f)_+\rangle
+
\left(  \frac{|X|}n \varkappa+ 
2\frac{\sqrt{|X|}}{\sqrt n}
\sqrt \varkappa 
+ 
\varkappa 
\right) 
\|f\|_2^2 
\quad 
\text{(by \cref{Lemma:InnerProdOfPlusPartWithPlusPartOfATranslate})} \\
& \leq 
(1- \delta) \|f_+\|_2^2
+
\left(  \frac{|X|}n \varkappa+ 
2\frac{\sqrt{|X|}}{\sqrt n}
\sqrt \varkappa 
+ 
\varkappa 
\right) 
\|f\|_2^2  \\
& \leq 
(1- \delta) 
\frac {1  + \varsigma + \varkappa + 2 \sqrt 2 \sqrt{\varkappa}} 2 \|f\|_2^2
+
\left(  \frac{|X|}n \varkappa+ 
2\frac{\sqrt{|X|}}{\sqrt n}
\sqrt \varkappa 
+ 
\varkappa 
\right) 
\|f\|_2^2  
\quad 
\text{(by \cref{Lemma:DiffOfL2NormOFf+f-})},
\end{align*}
which gives 
\begin{align*}
( 1- \varkappa)
\frac{|X|}n 
- 
2\frac{\sqrt{|X|}}{\sqrt n}
\sqrt \varkappa 
- 
\left(
(1- \delta) \frac {1  + \varsigma + \varkappa + 2 \sqrt 2 \sqrt{\varkappa}} 2 
+\varkappa
\right)
& \leq 0 .
\end{align*}
This yields
\begin{align*}
 \frac{\sqrt{|X|}}{\sqrt n }
& \leq 
\frac
{\sqrt \varkappa + \sqrt{\varkappa  + ( 1 - \varkappa) 
\left(
(1- \delta) \frac {1  + \varsigma + \varkappa + 2 \sqrt 2 \sqrt{\varkappa}} 2 
+\varkappa
\right)
}}
{1 - \varkappa}
.
\end{align*}
Consequently, we obtain 
\begin{align*}
\frac{| \supp(f_+) \cap \tau(\supp(f_-))|}n 
& = \frac{| \supp(f_+) \cap \tau(\supp(f_+))^c|} n \\
& = \frac{| \supp(f_+)|} n 
- \frac{| \supp(f_+) \cap \tau(\supp(f_+))|} n \\
& \geq  \frac 12
- \frac{| \supp(f_+) \cap \tau(\supp(f_+))|} n \\
& \geq  \frac 12 - 
\left(
\frac
{\sqrt \varkappa + \sqrt{\varkappa  + ( 1 - \varkappa) 
\left(
(1- \delta) \frac {1  + \varsigma + \varkappa + 2 \sqrt 2 \sqrt{\varkappa}} 2 
+\varkappa
\right)
}}
{1 - \varkappa}
\right)^2.
\end{align*}
For any subsets $A, B, C, D$ of $V$, note that 
\begin{align*}
&
((A\cap C) \cup (B \cap D))
\times  
((A\cap C) \cup (B \cap D))^c
\\
& \subseteq 
\bigl(
(A
\times  
B^c)
\cup 
(B
\times  
A^c)
\bigr)
\cup 
\bigl(
(C
\times  
D^c)
\cup 
(D
\times  
C^c)
\bigr)
\end{align*}
holds. 
It implies that 
\begin{align*}
& 
\langle T1_{
(\supp(f_+)\cap \tau(\supp(f_+)) ) \cup (\supp(f_-) \cap \varsigma_\tau (\supp(f_-)))
}, 
1_{
((\supp(f_+)\cap \tau(\supp(f_+)) ) \cup (\supp(f_-) \cap \varsigma_\tau (\supp(f_-))))^c }
\rangle \\
& \leq 
\langle T1_{\supp(f_+)}, 1_{\supp(f_-)^c}\rangle  
+ \langle T1_{\supp(f_-)}, 1_{\supp(f_+)^c} \rangle  \\
& + \langle T1_{\tau(\supp(f_+))} , 1_{\varsigma_\tau (\supp(f_-))^c} \rangle +\langle T1_{\varsigma_\tau (\supp(f_-))},  1_{\tau(\supp(f_+)) ^c}\rangle \\
& = 2(\langle T1_{\supp(f_+)}, 1_{\supp(f_-)^c}\rangle  + \langle T1_{\supp(f_-)}, 1_{\supp(f_+)^c} \rangle ) \\
& + \langle T1_{\tau(\supp(f_+))} , 1_{\varsigma_\tau (\supp(f_-))^c} \rangle +\langle T1_{\varsigma_\tau (\supp(f_-))},  1_{\tau(\supp(f_+)) ^c}\rangle \\
& = 2(\langle T1_{\supp(f_+)}, 1_{\supp(f_+)}\rangle  + \langle T1_{\supp(f_-)}, 1_{\supp(f_-)} \rangle ) , \\
& 
\langle T1_{
(\supp(f_+)\cap \tau(\supp(f_-)) ) \cup (\supp(f_-) \cap \varsigma_\tau (\supp(f_+)))
}, 
1_{
((\supp(f_+)\cap \tau(\supp(f_-)) ) \cup (\supp(f_-) \cap \varsigma_\tau (\supp(f_+))))^c }
\rangle \\
& \leq 
\langle T1_{\supp(f_+)}, 1_{\supp(f_-)^c}\rangle  + \langle T1_{\supp(f_-)}, 1_{\supp(f_+)^c} \rangle  \\
& + \langle T1_{\tau(\supp(f_-))} , 1_{\varsigma_\tau (\supp(f_+))^c} \rangle +\langle T1_{\varsigma_\tau (\supp(f_+))},  1_{\tau(\supp(f_-)) ^c}\rangle \\
& = 2(\langle T1_{\supp(f_+)}, 1_{\supp(f_-)^c}\rangle  + \langle T1_{\supp(f_-)}, 1_{\supp(f_+)^c} \rangle ) 
\quad 
\text{(by \cref{Left2Right})}\\
& = 2(\langle T1_{\supp(f_+)}, 1_{\supp(f_+)}\rangle  + \langle T1_{\supp(f_-)}, 1_{\supp(f_-)} \rangle ) .
\end{align*}
Note that the sets 
$$(\supp(f_+)\cap \tau(\supp(f_+)) ) \cup (\supp(f_-) \cap \varsigma_\tau (\supp(f_-)))
, $$
$$(\supp(f_+)\cap \tau(\supp(f_-)) ) \cup (\supp(f_-) \cap \varsigma_\tau (\supp(f_+)))
$$
are disjoint. 
Consequently, 
\begin{align*}
& \ecc_T d 
\min
\{
|(\supp(f_+)\cap \tau(\supp(f_+)) ) \cup (\supp(f_-) \cap \varsigma_\tau (\supp(f_-)))|
, \\
& \qquad 
|(\supp(f_+)\cap \tau(\supp(f_-)) ) \cup (\supp(f_-) \cap \varsigma_\tau (\supp(f_+)))|
\} \\
& \leq 2(\langle T1_{\supp(f_+)}, 1_{\supp(f_+)}\rangle  + \langle T1_{\supp(f_-)}, 1_{\supp(f_-)} \rangle )\\
& = 2d\beta n .
\end{align*}
Using \cref{Lemma:LowerBddForKappaGeneral}, we obtain 
\begin{align*}
(1-\mu_2) \frac {\varkappa}{1 - \varkappa}
& \geq 2 \beta \\
& \geq \ecc_T 
\min
\bigg\{
\frac
{|(\supp(f_+)\cap \tau(\supp(f_+)) ) \cup (\supp(f_-) \cap \varsigma_\tau (\supp(f_-)))|}n
, \\
& \qquad 
\frac{|(\supp(f_+)\cap \tau(\supp(f_-)) ) \cup (\supp(f_-) \cap \varsigma_\tau (\supp(f_+)))|}n
\bigg\} \\
& \geq \ecc_T 
\min
\bigg\{
\frac
{|(\supp(f_+)\cap \tau(\supp(f_+)) ) |}n
,
\frac{|(\supp(f_+)\cap \tau(\supp(f_-)) ) |}n
\bigg\} \\
& \geq \ecc_T
\min
\bigg\{
\left(
\frac {\delta \vartheta^2 - \varkappa}{\sqrt \varkappa + \sqrt{\varkappa + (1+ \varkappa) (\delta \vartheta^2 - \varkappa)}} 
\right)^2 
, \\
& \qquad 	
\frac 12 - 
\left(
\frac
{\sqrt \varkappa + \sqrt{\varkappa  + ( 1 - \varkappa) 
\left(
(1- \delta) \frac { 1  + \varsigma + \varkappa + 2 \sqrt 2 \sqrt{\varkappa}} 2 
+\varkappa
\right)
}}
{1 - \varkappa}
\right)^2
\bigg\}.
\end{align*}
If $\nu = 1$, we consider the case $\varkappa < \frac 1{2520}$, which implies that  
$$ (1- \mu_2) \frac \varkappa {1- \varkappa}
\geq 
\frac 2 {2520} \ecc_T
\geq \frac 1{2520}(1- \mu_2),
$$
and hence $\varkappa \geq \frac 1{2521}$. 
If $\nu \geq \frac 12$, we consider the case $\varkappa < \frac 1{54631}$, which gives the bound 
$$(1- \mu_2) \frac \varkappa {1- \varkappa}
\geq 
\frac 2 {54631} \ecc_T
\geq \frac 1{54631}(1- \mu_2),
$$
and hence $\varkappa \geq \frac1{54632}$. 
This completes the proof. 
\end{proof}

\section{Nowhere vanishing eigenfunctions}

In the following, we do not assume any of \cref{EigenFunction}, \cref{Condition22}, \cref{ConditionConn}. 
We assume only that \cref{SelfAdjConnNonBipartiteRegular}, and \cref{Left2Right} hold.

\begin{lemma}
\label{Lemma:NowhereVanishingEigenfunctions}
Suppose \cref{SelfAdjConnNonBipartiteRegular}, and \cref{Left2Right} hold.  
Let $d\tilde \mu$ be an eigenvalue of $T$ with 
$$ 1 + \tilde \mu < 1 - \mu_2.$$
Then the $d\tilde \mu$-eigenspace of $T$ contains an eigenfunction which is nonzero everywhere on $V$. 
Further, if $\varrho_\tau = \tau$ for all $\tau \in \calG$, then any eigenspace of $T$ contains an eigenfunction which is nonzero everywhere on $V$. 
\end{lemma}

\begin{proof}
Let $y$ be an eigenvector of $T$ with eigenvalue $d\tilde\mu  $. 
For any $\tau \in \calG$ and $u\in V$, it follows that 
\begin{align*}
\sum_v a_{uv} (y \circ \tau^\mo )(v)
& = \sum_v a_{vu} (y \circ \tau^\mo )(v) \\
& = \sum_v a_{\tau^\mo (v) \varrho_\tau ^\mo (u)} (y \circ \tau^\mo )(v) \\
& = \sum_v a_{ \varrho_\tau ^\mo (u) \tau^\mo (v)} (y \circ \tau^\mo )(v) \\
& = d\tilde\mu   y(\varrho_\tau ^\mo (u) )\\
& =d\tilde\mu    (y\circ \varrho_\tau ^\mo) (u),
\end{align*}
\begin{align*}
\sum_v a_{uv} (y\circ \varrho_\tau ^\mo )(v)
& = \sum_v a_{\tau^\mo(u) \varrho_\tau ^\mo (v)} (y\circ \varrho_\tau ^\mo )(v)\\
& = d\tilde\mu   (y\circ \tau ^\mo) (u),
\end{align*}
which implies that $T (y \circ \tau^\mo) = d\tilde\mu   y\circ \varrho_\tau ^\mo$, 
$T (y\circ \varrho_\tau ^\mo) = d\tilde\mu   y \circ \tau^\mo $. 
Note that $y \circ \tau^\mo - y\circ \varrho_\tau ^\mo$ is orthogonal to $u$. Further, if 
$y \circ \tau^\mo - y\circ \varrho_\tau ^\mo$ is nonzero for some $\tau\in \calG$, then it would be an eigenvector for $T$ with eigenvalue $-d\tilde\mu  $, implying that $-d\tilde\mu   \leq d\mu_2$, which yields 
$1 + \tilde\mu   \geq 1 - \mu_2$, which contradicts the hypothesis. 
It follows that the elements $y \circ \tau^\mo , y\circ \varrho_\tau ^\mo$ are equal for any $\tau \in \calG$. This shows that $y\circ \tau^\mo$ is also an eigenvector of $T$ with eigenvalue $d\tilde\mu  $. 

Hence, if $f$ is an eigenfunction of $T$ with eigenvalue $d\tilde\mu $, then for any $\tau\in \calG$,  $f(\tau \cdot)$ is also an eigenfunction with the same eigenvalue. 
If $f$ is nonzero at $v$, then for any $\tau\in \calG$, there is an eigenfunction $f_\tau := f(\tau \cdot)$ with eigenvalue $d\tilde\mu $ which is nonzero at $\tau^\mo v$. 
Further, note that if $f, g$ are eigenfunctions of $T$ with eigenvalue $d\tilde\mu $, there is an element $\lambda\in \bbR$ such that the support of $f+ \lambda g$ contains $\supp(f) \cup \supp (g)$. 
Since $\calG$ acts transitively on $V$, it follows that there exist real numbers $\{\lambda_\tau\}_{\tau \in \calG}$ such that 
the support of $\sum_\tau \lambda_\tau f_\tau$ contains $V$.

Note that if $\varrho_\tau = \tau$ for all $\tau \in \calG$, then from the above argument, it follows that for any eigenfunction $y$ of $T$, $y \circ \tau^\mo$ is also an eigenfunction of $T$ with the same eigenvalue. Then, taking a suitable linear combination of the elements of an eigenspace of $T$, we would obtain an eigenfunction which is nonzero everywhere. 
\end{proof}

\end{document}